\documentclass[preprint,12pt]{elsarticle}

\usepackage{amssymb}
\usepackage{amsmath}
\usepackage{stmaryrd}

\biboptions{compress}
\newtheorem{theorem}{Theorem}[section]
\newtheorem{lemma}[theorem]{Lemma}
\newtheorem{Corollary}[theorem]{Corollary}
\newtheorem{Proposition}[theorem]{Proposition}
\newdefinition{definition}[theorem]{Definition}
\newdefinition{remark}[theorem]{Remark}
\newdefinition{example}[theorem]{Example}
\newproof{proof}{Proof}

\journal{}

\begin{document}

\begin{frontmatter}



\title{The Algebraic Structure of Finitely Generated $L^{0}(\mathcal{F},K)$-Modules and the Helly Theorem in Random Normed Modules\tnoteref{t1}}

\tnotetext[t1]{Supported by NNSF No. 10871016}

\author[gt]{Tiexin Guo\corref{cor1}}
\ead{txguo@buaa.edu.cn}
\cortext[cor1]{Corresponding author}
\address[gt]{LMIB and School of Mathematics and Systems Science, Beihang University, Beijing 100191, P.R. China}

\author[gs]{Guang Shi}
\ead{g\_shi@ss.buaa.edu.cn}
\address[gs]{LMIB and School of Mathematics and Systems Science, Beihang University, Beijing 100191, P.R. China}

\begin{abstract}
Let $K$ be the scalar field of real numbers or complex numbers and $L^{0}(\mathcal{F},K)$ the algebra of equivalence classes of $K-$valued random variables defined on a probability space $(\Omega,\mathcal{F},P)$. In this paper, we first characterize the algebraic structure of finitely generated $L^{0}(\mathcal{F},K)$-modules and then combining the recently developed separation theorem in random locally convex modules
we prove the Helly theorem in random normed modules with the countable concatenation property under the framework of random conjugate spaces at the same time a simple counterexample shows that it is necessary to require the countable concatenation property. By the way,we also give an application to the existence problem of the random solution of a system of random linear functional equations.

\end{abstract}


\end{frontmatter}


\section{Introduction and main results}
\label{}
In 1942, K. Menger thought that the distance between two points in a real space is random and so he presented a probabilistic generalization of a classical metric space, namely the notion of a probabilistic metric space (briefly, a $PM$ space) in which the distance between two points is described by a probability distribution function.  Subsequently, the theory of $PM$ spaces was founded and deeply developed by B. Schweizer and A. Sklar \cite{SS}. Following K. Menger's idea, A. N. Serstnev presented the notion of a probabilistic normed space (briefly, a $PN$ space) in 1962, then in 1993 C. Alsina, B. Schweizer and A. Sklar redefined $PN$ spaces in a more general way in \cite{CBA1} and in 1997 they and C. Sempi presented the notion of a probabilistic inner product space (briefly, a $PIP$ space) in \cite{CBA2}. $PN$ spaces are usually endowed with a natural topology, called the $(\varepsilon,\lambda)$-topology, so that they are metrizable linear topological spaces under a mild condition \cite{CBA3}, see \cite{BJC1,BC,CS} for the closely related studies of $PN$ spaces. Since $PN$ spaces are rarely locally convex spaces, for example, Menger $PN$ spaces under a $t$-norm other than the $t$-norm Min are not locally convex spaces in general, even they do not admit a nontrival continuous linear functional, and so the theory of traditional conjugate spaces universally fails to serve for the deep development of $PN$ spaces. Considering the fundamental importance of the theory of traditional conjugate spaces in functional analysis, a natural problem is: whether does there exist a proper probabilistic generalization of the theory of traditional conjugate spaces which perfectly matches the theory of general $PN$ spaces? As stated in \cite{SS1}, this problem is still an open and challenging problem. Actually, in the last 15 years the development of random metric theory is closely related to this problem, in fact, this problem has been satisfactorily solved within random metric theory \cite{rela}.

Random metric theory originated from the theory of probabilistic metric spaces. The random distance between two points in an original random metric space (briefly, an $RM$ space) is a nonnegative random variable defined on some probability space \cite[Chapter 9]{SS}, similarly, the random norm of a vector in an original random normed space (briefly, an $RN$ space) is a nonnegative random variable defined on some probability space \cite[Chapter 15]{SS}. The development of $RN$ spaces in the direction of functional analysis led us to present the new versions of an $RM$ and $RN$ space in \cite{base}, where the random distances or random norms are defined to be the equivalence classes of nonnegative random variables according to the new versions. Since an $RN$ space under the $(\varepsilon,\lambda)$-topology is not a locally convex space in general, the theory of traditional conjugate spaces universally fails to serve for the theory of $RN$ spaces. Based on the new version of an $RN$ space we presented a definitive definition of the random conjugate space for an $RN$ space, further the deep development of the theory of random conjugate spaces led us to present the notion of a random normed module (briefly, an $RN$ module) in \cite{base}, which is the elaboration of the notion of the original $RN$ module introduced in \cite{exte}. With the notions of $RN$ modules and their random conjugate spaces at hand, we have developed deeply and systematically the theory of $RN$ modules under the $(\varepsilon,\lambda)$-topology \cite{Radon,Riesz,dual,james,bana}. An interesting phenomenon is: some classical theorems such as the Riesz's representation theorem in Hilbert spaces and the James theorem in Banach spaces still hold in complete random inner product modules (briefly, $RIP$ modules) and complete $RN$ modules, respectively \cite{Riesz, james}, whereas the others such as the classical Banach-Alaoglu theorem and Banach-Bourbaki-Kakutani-$\check{\textmd{S}}$mulian theorem do not universally hold in our random setting \cite{bana}.

The classical Helly theorem \cite{Helly}, as one of the basic theorems in functional analysis, is a powerful tool for the study of a system of linear functional equations, which says that if $X$ is a normed space over the scalar field $K$, $f_{1},f_{2},\cdots, f_{n}$ are any given bounded linear functionals on $X$, $\alpha_{1},\alpha_{2},\cdots, \alpha_{n}$ are any given constants in $K$ and $\beta$ is any given nonnegative number, then for any positive number $\varepsilon$ there exists $x_{\varepsilon}\in X$ such that the following conditions are satisfied

(1) $f_{i}(x_{\varepsilon})=\alpha_{i}$ for any $i$ such that $1\leqslant i\leqslant n$,

(2) $\|x_{\varepsilon}\|\leqslant \beta+\varepsilon$

\noindent if and only if $|\Sigma_{i=1}^{n}\lambda_{i}\alpha_{i}|\leqslant \beta\|\Sigma_{i=1}^{n}\lambda_{i}f_{i}\|$ holds for all $\lambda_{1},\lambda_{2},\cdots,\lambda_{n}\in K$.

Then, does the interesting Helly theorem hold in $RN$ modules under the framework of random conjugate spaces? The purpose of this paper is to give an affirmative answer for $RN$ modules with the countable concatenation property. The answer to this problem will involve both the algebraic characterization of finitely generated $L^{0}(\mathcal{F},K)$-Modules (see Theorem \ref{thm1} below) and the recently developed separation theorem in random locally convex modules under the locally $L^{0}$-convex topology (see Lemma \ref{lem6} below).

The notion of a random locally convex module was first introduced in \cite{pro} and deeply developed under the $(\varepsilon,\lambda)$-topology in \cite{acha,sepa} for the further development of the theory of $RN$ modules. In 2009, motivated by financial applications, Filipovi\'{c}, Kupper and Vogelpoth presented in \cite{DMN} a new topology (called the locally $L^{0}$-convex topology) for a random locally convex module and proved that the theory of a Hausdorff locally $L^{0}$-convex module introduced in \cite{DMN} is equivalent to the theory of a random locally convex module endowed with the locally $L^{0}$-convex topology. Subsequently, the relations between some basic results derived from the two kinds of topologies were studied in \cite{rela}. Now, random locally convex modules and in particular random normed modules together with their random conjugate spaces have been a proper framework for $L^{0}$-convex analysis playing a crucial role in the study of conditional risk measures \cite{rece}. As shown in \cite{rece}, the $(\varepsilon,\lambda)$-topology and the locally $L^{0}$-convex topology have their respective advantages and disadvantages and they can complement each other in the study of random locally convex modules, for example, the $(\varepsilon,\lambda)$-topology is very natural but too weak to ensure that a random locally convex module has even an $L^{0}$-convex open proper subset, whereas the locally $L^{0}$-convex topology is too strong but can guarantee that a random locally convex module has rich $L^{0}$-convex open subsets so that Filipovi\'{c}, Kupper and Vogelpoth can prove a separation theorem \cite[Theorem 2.6]{DMN} between the two $L^{0}$-convex subsets if either of them is open in the locally $L^{0}$-convex topology, their result was further generalized to a more general form \cite[Theorem 3.15]{rela}. Besides a variant of \cite[Theorem 3.15]{rela} (see Lemma \ref{lem6} below), the following key result, namely Theorem 1.1 below, concerning the algebraic characterization of the structure of finitely generated $L^{0}(\mathcal{F},K)$-modules is in particular crucial in this paper.

To introduce the two main results of this paper, we first recall some notation and terminology as follows. Throughout the paper, $K$ always denotes the scalar field $R$ of real numbers or $C$ of complex numbers, $(\Omega,\mathcal{F},P)$ a probability space and $L^{0}(\mathcal{F},K)$ the algebra of equivalence classes of $K$-valued random variables on $(\Omega,\mathcal{F},P)$ under the ordinary addition, multiplication and scalar multiplication operations on equivalence classes. A left module $E$ over the algebra $L^{0}(\mathcal{F},K)$ (briefly, an $L^{0}(\mathcal{F},K)$-module) is called finitely generated if there exist finitely many elements $x_{1},x_{2},\cdots,x_{n}$ in $E$ such that $E=\{\sum_{i=1}^{n}\xi_{i}x_{i}~|~\xi_{i}\in L^{0}(\mathcal{F},K),1\leqslant i\leqslant n\}$. Obviously, for an $\mathcal{F}$-measurable subset $A$ of $\Omega$ and an $L^{0}(\mathcal{F},K)$-module $E$, $\tilde{I}_{A}E:=\{\tilde{I}_{A}x~|~x\in E\}$, called the $A$-stratification of $E$, is a left module over the algebra $\tilde{I}_{A}L^{0}(\mathcal{F},K):=\{\tilde{I}_{A}\xi~|~\xi\in L^{0}(\mathcal{F},K)\}$, where $\tilde{I}_{A}$ is the equivalence class determined by the characteristic function of $A$. Further, for an $\mathcal{F}$-measurable subset $A$ of $\Omega$,  an $L^{0}(\mathcal{F},K)$-module $E$ is said to be free on $A$ if the $A$-stratification of $E$ is free over the algebra $\tilde{I}_{A}L^{0}(\mathcal{F},K)$.

A finite partition $\{A_{0},A_{1},\cdots,A_{n}\}$ of $\Omega$ to $\mathcal{F}$ means that $A_{i}\in \mathcal{F}$, $A_{i}\cap A_{j}=\emptyset(i\neq j)$ for any $i$ and $j$ in $\{0,1,2,\cdots,n\}$ and
$\Omega=\bigcup_{i=0}^{n}A_{i}$. Obviously, for any finite partition $\{A_{0},A_{1},\cdots,A_{n}\}$ of $\Omega$ to $\mathcal{F}$, it always holds that an $L^{0}(\mathcal{F},K)$-module $E$ equals $\bigoplus_{i=0}^{n}\tilde{I}_{A_{i}}E$, where the right side stands for the direct sum of the submodules $\tilde{I}_{A_{i}}E$. But for a finitely generated $L^{0}(\mathcal{F},K)$-module $E$, we can find a useful direct sum decomposition as our first main result of this paper --- Theorem 1.1 below --- exhibits:

\begin{theorem}\label{thm1}
Let $E$ be a finitely generated $L^{0}(\mathcal{F},K)$-module. Then there exists a finite partition $\{A_{0},A_{1},\cdots,A_{n}\}$ of $\Omega$ to $\mathcal{F}$ such that $\tilde{I}_{A_{i}}E$ is a free module of rank $i$ over the algebra $\tilde{I}_{A_{i}}L^{0}(\mathcal{F},K)$ for each $i\in \{0,1,2,\cdots,n\}$ satisfying $P(A_{i})>0$, in which case $E=\bigoplus_{i=0}^{n}\tilde{I}_{A_{i}}E$ and each such $A_{i}$ is unique in the sense of almost sure equality.
\end{theorem}

To introduce our second main result of this paper, we first recall from \cite{line}:
Let $\bar{L}^{0}(\mathcal{F},R)$ be the set of equivalence classes of extended real-valued random variables on $(\Omega,\mathcal{F},P)$, then $\bar{L}^{0}(\mathcal{F},R)$ is partially ordered by $\xi \leqslant\eta$ if and only if $\xi^{0}(\omega) \leqslant \eta^{0}(\omega)$ for $P$--almost all $\omega$ in $\Omega$ (briefly, a.s.), where $\xi^{0}$ and $\eta^{0}$ are arbitrarily chosen representatives of $\xi$ and $\eta$ in $\bar{L}^{0}(\mathcal{F},R)$, respectively. Furthermore, every subset $H$ of $\bar{L}^{0}(\mathcal{F},R)$ has a supremum and infimum, denoted by $\bigvee H$ and $\bigwedge H$, respectively. It is also well known from \cite{line} that $L^{0}(\mathcal{F},R)$, as a sublattice of $\bar{L}^{0}(\mathcal{F},R)$,  is a complete lattice in the sense that every subset with an upper bound has a supremum.

As usual, $\xi >\eta$ means $\xi \geqslant \eta$ and $\xi \neq \eta$, whereas $\xi >\eta$ on $A$
means $\xi^{0}(\omega)>\eta^{0}(\omega)$ a.s. on $A$ for any $A\in \mathcal{F}$ and $\xi$ and $\eta$ in $\bar{L}^{0}(\mathcal{F},R)$, where $\xi^{0}$ and $\eta^{0}$ are arbitrarily chosen representatives of $\xi$ and $\eta$, respectively.

Specially, we denote $L^{0}_{+}(\mathcal{F})=\{\xi \in L^{0}(\mathcal{F},R)~|~ \xi \geqslant 0\}$ and $L^{0}_{++}(\mathcal{F})=\{\xi \in L^{0}(\mathcal{F},R)~|~\xi > 0 ~ \mbox{on} ~ \Omega\}$.

\begin{definition}[\cite{base, rela}]
An ordered pair $(E,\|\cdot\|)$ is called an $RN$ module over $K$ with base $(\Omega,\mathcal{F},P)$ if $E$ is a left module over the algebra $L^{0}(\mathcal{F},K)$ and $\|\cdot \|$ is a mapping from $E$ to $L^{0}_{+}(\mathcal{F})$ such that the following three axioms are satisfied:

\renewcommand{\labelenumi}{$($\arabic{enumi}$)$}
\begin{enumerate}

\item $\|x\|=0$ if and only if $x=\theta$ (the null element of $E$);

\item $\|\xi x\|=|\xi|\|x\|$, $\forall \xi \in L^{0}(\mathcal{F},K)$ and $\forall x \in E$;

\item $\|x+y\|\leqslant \|x\|+\|y\|$, $\forall x,y \in E$,
\end{enumerate}
where the mapping $\|\cdot\|$ is called the $L^{0}$-norm on $E$ and $\|x\|$ is called the $L^{0}$-norm of a vector $x\in E$. Besides, a mapping $\|\cdot \|:E\rightarrow L^{0}_{+}(\mathcal{F})$ only satisfying (2) and (3) above is called an $L^{0}$-seminorm on $E$.
\end{definition}

\begin{definition}[\cite{base,rela}]
A linear operator $f$ from an $RN$ module $(E,\|\cdot\|)$ over $K$ with base $(\Omega,\mathcal{F},P)$ to $L^{0}(\mathcal{F},K)$ is called an a.s. bounded random linear functional on $E$ if there exists some $\xi$ in $L^{0}_{+}(\mathcal{F})$ such that $|f(x)|\leqslant \xi \|x\|$, $\forall x \in E$. Let $E^{*}$ be the linear space of a.s. bounded random linear functionals on $E$, further define the module multiplication $\cdot:L^{0}(\mathcal{F},K)\times E^{*}\rightarrow E^{*}$ by $(\xi \cdot f)(x)=\xi(f(x)),\forall \xi\in L^{0}(\mathcal{F},K),  f\in E^{*}$ and $x\in E$, and the mapping $\|\cdot\|^{*}:E^{*}\rightarrow L^{0}_{+}(\mathcal{F})$ by $\|f\|^{*}=\bigwedge\{\xi\in L^{0}_{+}(\mathcal{F})~|~|f(x)|\leqslant\xi \|x\|,\forall x\in E\}$, $\forall f\in E^{*}$, then it is easy to see that $(E^{*},\|\cdot\|^{*})$ is an $RN$ module over $K$ with base $(\Omega,\mathcal{F},P)$, called the random conjugate space of $(E,\|\cdot\|)$.
\end{definition}

\begin{definition}[\cite{rela}]
Let $E$ be a left module over the algebra $L^{0}(\mathcal{F},K)$. A formal sum $\Sigma_{n\in N}\tilde{I}_{A_{n}}x_{n}$ is called a countable concatenation of a
sequence $\{x_{n},n\in N\}$ in $E$ with respect to a countable partition $\{A_{n},n\in N\}$ of $\Omega$ to $\mathcal{F}$. Moreover, a countable concatenation
$\Sigma_{n\in N}\tilde{I}_{A_{n}}x_{n}$ is well defined or $\Sigma_{n\in N}\tilde{I}_{A_{n}}x_{n}\in E$
if there is $x\in E$ such that $\tilde{I}_{A_{n}}x=\tilde{I}_{A_{n}}x_{n}$, $\forall n\in N$.
A subset $G$ of $E$ is called having the countable concatenation property if every countable
concatenation $\Sigma_{n\in N}\tilde{I}_{A_{n}}x_{n}$ with $x_{n}\in G$ for each $n\in N$
still belongs to $G$, namely $\Sigma_{n\in N}\tilde{I}_{A_{n}}x_{n}$ is well defined and there
exists $x\in G$ such that $x=\Sigma_{n\in N}\tilde{I}_{A_{n}}x_{n}$.
\end{definition}

We can now state our second main result as follows:

\begin{theorem}\label{thm2}
Let $(E,\|\cdot\|)$ be an $RN$ module over $K$ with base $(\Omega,\mathcal{F},P)$ such that $E$ has
the countable concatenation property, $\{f_{1},f_{2},\cdots,f_{n} \}\subset E^{*}$,
$\{\xi_{1}, \xi_{2},\cdots ,\xi_{n} \}\subset L^{0}(\mathcal{F},K)$ and
$\beta \in L^{0}_{+}(\mathcal{F})$.
Then for any $\varepsilon \in L^{0}_{++}(\mathcal{F})$ there exists $x_{\varepsilon} \in E$ such that the following two conditions are satisfied

(1) $f_{i}(x_{\varepsilon})=\xi _{i}$, $i=1,2,\cdots ,n$;

(2) $\|x_{\varepsilon}\|\leqslant \beta+\varepsilon$

\noindent if and only if
$|\sum_{k=1}^{n}\lambda_{k}\xi_{k}|\leqslant \beta\|\sum_{k=1}^{n} \lambda_{k}f_{k}\|^{*}$ holds for any $\lambda_{1},\lambda_{2},\cdots \lambda_{n}\in L^{0}(\mathcal{F},K)$.
\end{theorem}

The remainder of this paper is organized as follows: in Section 2 we give the proof of Theorem \ref{thm1} together with its two interesting corollaries and in Section 3 we give the proof of Theorem \ref{thm2}.

\section{The algebraic structure of finitely generated $L^{0}(\mathcal{F},K)$-modules}
\label{}

For the proof of Theorem \ref{thm1}, we need another formulation of completeness of the lattice $L^{0}(\mathcal{F},R)$. Let us first recall the notions of an essential supremum and infimum of a set of real-valued random variables from \cite{He-Wang-Yan}: let $\bar{\mathcal{L}}^{0}(\mathcal{F},R)$ be the set of extended real-valued random variables on $(\Omega,\mathcal{F},P)$ and $H$ a subset of $\bar{\mathcal{L}}^{0}(\mathcal{F},R)$, $\xi\in \bar{\mathcal{L}}^{0}(\mathcal{F},R)$ is called an essential upper bound for $H$ if $\eta(\omega)\leqslant \xi(\omega)$ $P$-a.s. for any $\eta \in H$, in addition if $\xi(\omega)\leqslant \xi^{'}(\omega)$ $P$-a.s. for each essential upper bound $\xi^{'}$ for $H$, then the essential upper bound $\xi$ is called an essential supremum for $H$. Similarly, one can have the notion of an essential infimum. It is well known that every subset $H$ of $\bar{\mathcal{L}}^{0}(\mathcal{F},R)$ has an essential supremum and infimum, denoted by esssup$H$ and essinf$H$, respectively, and they are unique in the sense of $P$-a.s. equality. Furthermore, if $H$ is also directed upwards (downwards) then there exists a nondecreasing (nonincreasing) sequence $\{a_{n},n\in N\}$ (resp., $\{b_{n},n\in N\}$) in $H$ such that esssup$H=$esssup$\{a_{n},n\in N\}$ (resp., essinf$H=$essinf$\{b_{n},n\in N\}$).

It is easy to see from above that for a nonempty subfamily $\mathcal{E}$ of $\mathcal{F}$ there uniquely exist $A$ and $B$ in $\mathcal{F}$ in the sense of P-a.s. equality such that $I_{A}=$esssup$\{I_{E}~|~E\in\mathcal{E}\}$ and $I_{B}=$essinf$\{I_{E}~|~E\in\mathcal{E}\}$, such an $A$ and $B$ are called an essential supremum and infimum of $\mathcal{E}$, denoted by esssup$\mathcal{E}$ and essinf$\mathcal{E}$, respectively.

In the sequel of this paper, we make the following convention: if $I_{A}$ denotes the characteristic function of an $\mathcal F$--measurable set $A$, then we always use $\tilde{I}_{A}$ for its equivalence class.

Besides, for any $\xi ,\eta\in \bar{L}^{0}(\mathcal{F},R)$, $[\xi >\eta]$ denotes the equivalence class of the $\mathcal F$-measurable set $\{\omega\in \Omega~|~\xi^{0}(\omega)>\eta^{0}(\omega)\}$ and $I_{[\xi >\eta]}$ the equivalence class of the characteristic function of $\{\omega\in \Omega~|~\xi^{0}(\omega)>\eta^{0}(\omega)\}$, where $\xi^{0}$ and $\eta^{0}$ are arbitrarily chosen representatives of $\xi$ and $\eta$, respectively. Further, for any $\xi \in L^{0}(\mathcal{F},K)$, $\xi^{-1}$ stands for the equivalence class of the $\mathcal{F}$-measurable function $(\xi^{0})^{-1}:\Omega\rightarrow K$ defined by
$$
(\xi^{0})^{-1}(\omega)=\left \{
\begin{array}{ll}
(\xi^{0}(\omega))^{-1}, &\mbox{if~}\xi^{0}(\omega)\neq 0;\\
0, &\mbox{otherwise},
\end{array}
\right.
$$ where $\xi^{0}$ is an arbitrarily chosen representative of $\xi$. It is clear that
and $\xi \cdot \xi^{-1}=I_{[|\xi|>0]}$.

Finally, for any $A\in \mathcal{F}$ we always use $\tilde{A }$ for the equivalence class of $A$, namely $\tilde{A }=\{B\in \mathcal{F}~|~P(A\bigtriangleup B)=0\}$, where $\bigtriangleup$ denotes the symmetric difference of $A$ and $B$. Further, let $\widetilde{\mathcal{F}}=\{\tilde{A}~|~A\in \mathcal{F}\}$, we make the following convention: $P(\tilde{A})=P(A), \forall A\in\mathcal{F}$, and define $\tilde{A}\setminus \tilde{B}=$ the equivalence class of $A\setminus B$ and $\tilde{A}\cup \tilde{B}=$ the equivalence class of $A\cup B$ for any $A$ and $B$ in $\mathcal{F}$.

The proof of Theorem \ref{thm1} needs Lemmas \ref{lem1}, \ref{lem2} and \ref{lem3} below.

\begin{lemma}\label{lem1}
Let $A$ be an $\mathcal{F}$-measurable subset of $\Omega$ such that $P(A)>0$, $m$ and $h$ positive integers and $\xi_{ij}\in \tilde{I}_{A}L^{0}(\mathcal{F},K)$ for $1\leqslant i\leqslant m$ and $1\leqslant j\leqslant h$. If $h>m$, then the following system of linear equations
 \begin{equation}\left\{
\begin{array}{ccccccccc}
\xi_{11}\lambda_{1}&+&\xi_{12}\lambda_{2}&+&\cdots&+&\xi_{1h}\lambda_{h}&=&0\\
\xi_{21}\lambda_{1}&+&\xi_{22}\lambda_{2}&+&\cdots&+&\xi_{2h}\lambda_{h}&=&0\\
\multicolumn{9}{c}{\dotfill}\\
\xi_{m1}\lambda_{1}&+&\xi_{m2}\lambda_{2}&+&\cdots&+&\xi_{mh}\lambda_{h}&=&0\\
\end{array}
\right
.\end{equation}
has a nontrivial solution $(\lambda_{1},\lambda_{2},\cdots,\lambda_{h})$ in $\tilde{I}_{A}L^{0}(\mathcal{F},K^{h})$, where $L^{0}(\mathcal{F},K^{h})$ is the set of equivalence classes of $K^{h}$-valued random variables on $(\Omega,\mathcal{F},P)$.

\end{lemma}

\begin{proof}
With no loss of generality, suppose $A=\Omega$. If $P(\tilde{\Omega}\setminus\bigcup_{j=1}^{m}[|\xi_{j1}|>0])>0$, let $\lambda_{1}=I_{\tilde{\Omega}\setminus\bigcup_{j=1}^{m}[|\xi_{j1}|>0]}$ and $\lambda_{i}=0$ for $2\leqslant i \leqslant h$, then it is easy to verify that $(\lambda_{1},\lambda_{2},\cdots,\lambda_{h})$ is a desired solution. Otherwise, $P(\tilde{\Omega}\setminus\bigcup_{j=1}^{m}[|\xi_{1j}|>0])=0$, let $C_{l}=[|\xi_{l1}|>0]\setminus\bigcup_{j=1}^{l-1}[|\xi_{j1}|>0]$ for $2\leqslant l \leqslant m$ and $\eta_{1i}=(I_{\tilde{\Omega}\setminus[|\xi_{11}|>0]}+\xi_{11}^{-1})\xi_{1i}+\sum_{l=2}^{m}I_{C_{l}}\xi_{l1}^{-1}\xi_{li}$ for $1\leqslant i \leqslant h$. In addition, let $\eta_{ji}=\xi_{ji}-\xi_{j1}^{-1}\eta_{1i}$ for $1\leqslant i \leqslant h$ and $2\leqslant j \leqslant m$. Clearly $\eta_{11}=I_{\tilde{\Omega}}$ and $\eta_{j1}=0$ for $2\leqslant j \leqslant m$. It is easy to check that the following system of linear equations
$$
\left\{
\begin{array}{ccccccccc}
\lambda_{1}&+&\eta_{12}\lambda_{2}&+&\cdots&+&\eta_{1h}\lambda_{h}&=&0\\
                    & &\eta_{22}\lambda_{2}&+&\cdots&+&\eta_{2h}\lambda_{h}&=&0\\
        \multicolumn{9}{c}{\dotfill}\\
                    & &\eta_{m2}\lambda_{2}&+&\cdots&+&\eta_{mh}\lambda_{h}&=&0\\
\end{array} \right
.$$
is equivalent to (2.1).

By induction method, we can eventually obtain either a desired solution or the following system of linear equations
$$
\left\{
\begin{array}{cccccccccccccc}
\lambda_{1}& & & & &+&\beta_{1,m+1}\lambda_{m+1}&+&\cdots& +&\beta_{1h}\lambda_{h}&=&0\\
                    & &\lambda_{2}& & &+&\beta_{2,m+1}\lambda_{m+1}&+&\cdots& +&\beta_{2h}\lambda_{h}&=&0\\
\multicolumn{9}{c}{\dotfill}\\
                    & & & &\lambda_{m}&+&\beta_{m,m+1}\lambda_{m+1}&+&\cdots&+&\beta_{mh}\lambda_{h}&=&0\\
\end{array} \right
.$$
which is still equivalent to (2.1). Let $\lambda_{i}=\beta_{i,m+1}$ for $1\leqslant i \leqslant m$, $\lambda_{m+1}=-I_{\tilde{\Omega}}$ and $\lambda_{j}=0$ for $m+2\leqslant j \leqslant h$, then such $(\lambda_{1},\lambda_{2},\cdots,\lambda_{h})$ satisfies our requirement.
\end{proof}

\begin{lemma}\label{lem2}
Let $E$ be an $L^{0}(\mathcal{F},K)$-module and $\mathcal{A}_{i}=\{A\in \mathcal{F}~|~ \tilde{I}_{A}E$ is a free module of rank $i$ over the algebra $\tilde{I}_{A}L^{0}(\mathcal{F},K)\}$ for some nonnegative integer $i$. If $\mathcal{A}_{i}\neq\emptyset$, then esssup$\mathcal{A}_{i}\in \mathcal{A}_{i}$.
\end{lemma}

\begin{proof}
Suppose $\mathcal{A}_{i}\neq\emptyset$ for some $1\leqslant i \leqslant n$, $A,B\in \mathcal{A}_{i}$ and $\{y_{1},y_{2},\cdots,y_{i}\},\{y_{1}^{'}$, $y_{2}^{'},\cdots,y_{i}^{'}\}$ are bases for the $\tilde{I}_{A}L^{0}(\mathcal{F},K)$-module $\tilde{I}_{A}E$ and the $\tilde{I}_{B}L^{0}(\mathcal{F},K)$-module $\tilde{I}_{B}E$, respectively, then $\{\tilde{I}_{A}y_{1}+\tilde{I}_{A\setminus B}y_{1}^{'},\tilde{I}_{A}y_{2}+\tilde{I}_{A\setminus B}y_{2}^{'},\cdots,\tilde{I}_{A}y_{i}+\tilde{I}_{A\setminus B}y_{i}^{'}\}$ is a basis for the $\tilde{I}_{A\cup B}L^{0}(\mathcal{F},K)$-module $\tilde{I}_{A\cup B}E$, i.e. $A\cup B\in \mathcal{A}_{i}$. Actually, let $z_{j}=\tilde{I}_{A}y_{j}+\tilde{I}_{A\setminus B}y_{j}^{'}$ for $1 \leqslant j\leqslant i$. If $\sum_{j=1}^{i}\lambda_{j}z_{j}=\theta$ for some $\lambda_{j}\in \tilde{I}_{A\cup B}L^{0}(\mathcal{F},K), 1 \leqslant j\leqslant i$, then $\sum_{j=1}^{i}\tilde{I}_{A}\lambda_{j}y_{j}=\tilde{I}_{A}\sum_{j=1}^{i}\lambda_{j}z_{j}=\theta$, which implies $\tilde{I}_{A}\lambda_{j}=0$ for $1 \leqslant j\leqslant i$. Likewise, we have $\tilde{I}_{B}\lambda_{j}=0$ for $1 \leqslant j\leqslant i$, and it follows that $\lambda_{j}=0$ for $1 \leqslant j\leqslant i$. Thus $\{z_{1},z_{2},\cdots,z_{i}\}$ is $\tilde{I}_{A\cup B}L^{0}(\mathcal{F},K)$-independent. Furthermore, for any $x\in \tilde{I}_{A\cup B}E$ there exist $\{\xi_{j}\}_{j=1}^{i}\subset \tilde{I}_{A}L^{0}(\mathcal{F},K)$ and $\{\eta_{j}\}_{j=1}^{i}\subset \tilde{I}_{B}L^{0}(\mathcal{F},K)$ such that $\tilde{I}_{A}x=\sum_{j=1}^{i}\xi_{j}y_{j}$ and $\tilde{I}_{ B}x=\sum_{j=1}^{i}\eta_{j}y^{'}_{j}$. Hence
$$x=\tilde{I}_{A}x+\tilde{I}_{A\setminus B}x=\sum_{j=1}^{i}\xi_{j}y_{j}+\tilde{I}_{A\setminus B}\sum_{j=1}^{i}\eta_{j}y^{'}_{j}=\sum_{j=1}^{i}(\xi_{j}+\tilde{I}_{A\setminus B}\eta_{j})z_{j},$$ which proves our claim.

Thus $\mathcal{A}_{i}$ is directed upwards. Let $A_{i}=$esssup$\mathcal{A}_{i}$, then there exists a nondecreasing sequence $\{B_{k},k\in N\}$ in $\mathcal {A}_{i}$
such that $A_{i}=\bigcup _{k\in N} B_{k}$. Let $\{C_{k},k\in N\}$ be a sequence of $\mathcal{F}$-measurable sets such that $C_{1}=B_{1}$, $C_{k}=B_{k}\setminus B_{k-1}$ for $k>1$ and $N^{'}:=\{k\in N~|~P(C_{k})>0\}$. Clearly $C_{k}\in \mathcal{A}_{i}$ for each $k \in N^{'}$, thus there exists $\{y_{j}^{k}~|~1\leqslant j\leqslant i\}\subset \tilde{I}_{C_{k}}E$ for each
$k\in N^{'}$ such that $\{y_{j}^{k}~|~1\leqslant j\leqslant i\}$ is a basis for $\tilde{I}_{C_{k}}E$. Since $E$ is finitely generated, $E$ has the countable concatenation property, then $y_{j}:=\sum_{k\in N^{'}}\tilde{I}_{C_{k}}y_{j}^{k}$ belongs to $E$ for $1\leqslant j\leqslant i$. It is easy to check that $\{y_{j}\}_{j=1}^{i}$ is a basis for $\tilde{I}_{A_{i}}E$, which implies $A_{i}\in \mathcal{A}_{i}$. Finally, let $\mathcal{A}_{0}=\{A\in \mathcal{F}~|~\tilde{I}_{A}E=\{\theta\}\}$ and $A_{0}=$esssup$\mathcal{A}_{0}$, then it is easy to verify that  $A_{0}\in \mathcal{A}_{0}$, which completes the proof.
\end{proof}

For the sake of the reader's convenience, we give the notion of the countable concatenation hull of a subset of an $L^{0}(\mathcal{F},K)$-module, which was first introduced in \cite{rela}.

\begin{definition}[\cite{rela}]
Let $E$ be an $L^{0}(\mathcal{F},K)$-module and $G$ a subset of $E$. The set of countable concatenations $\Sigma_{n\in N}\tilde{I}_{A_{n}}x_{n}$ with $x_{n}\in G$ for each $n\in N$ is called the countable concatenation hull of $G$, denoted by $H_{cc}(G)$.
\end{definition}

Clearly, we have $H_{cc}(G)\supset G$ for any subset $G$ of an $L^{0}(\mathcal{F},K)$-module $E$, and $G$ has the countable concatenation property if and only if $H_{cc}(G)=G$. For the proof of Theorem \ref{thm1}, we still need the following:

\begin{lemma}[{\cite[Theorem 3.13]{rela}}]\label{lem3}
Let $E$ be an $L^{0}(\mathcal{F},K)$-module and $G$ and $M$ any two nonempty subsets of $E$ such that $\tilde{I}_{A}G+\tilde{I}_{A^{c}}G\subset G$ and $\tilde{I}_{A}M+\tilde{I}_{A^{c}}M\subset M$ for any $A\in\mathcal{F}$. If $H_{cc}(G)\cap H_{cc}(M)=\emptyset$, then there exists a set $H(G,M)\in \mathcal{F}$, which is unique in the sense of equivalence, called the hereditarily disjoint stratification of $G$ and $M$, such that the following are satisfied:

\begin{enumerate}
\item $P(H(G,M))>0$;

\item $\tilde{I}_{A}G\cap \tilde{I}_{A}M=\emptyset$ for all $A\in \mathcal{F}, A\subset H(G,M)$ with $P(A)>0$;

\item $\tilde{I}_{A}G\cap \tilde{I}_{A}M\neq\emptyset$ for all $A\in \mathcal{F}, A\subset \Omega\setminus H(G,M)$ with $P(A)>0$.
\end{enumerate}
\end{lemma}

We can now prove Theorem 1.1.

\newproof{pot}{Proof of Theorem \ref{thm1}}
\begin{pot}

Suppose $E$ is generated by $\{x_{1},x_{2},\cdots,x_{n}\}\subset E$ and $\mathcal{A}_{i}$ is defined as in Lemma \ref{lem2} for $0\leqslant i\leqslant n$. Let $A_{i}=$esssup$\mathcal{A}_{i}$ for each $i$ such that $0\leqslant i \leqslant n$ and $\mathcal{A}_{i}\neq \emptyset$, then it follows from Lemma \ref{lem2} that $\tilde{I}_{A_{i}}E$ is a free module of rank $i$ over the algebra $\tilde{I}_{A_{i}}L^{0}(\mathcal{F},K)$. Besides, let $A_{i}=\emptyset$ if $0\leqslant i \leqslant n$ and $\mathcal{A}_{i}= \emptyset$.

Let $B=\Omega \setminus \bigcup_{0\leqslant k \leqslant n}A_{k}$, then we claim that $P(B)=0$. Otherwise, there exists $y_{1}\in \tilde{I}_{B}E$ such that $y_{1}\neq \theta$ since $B\notin \mathcal{A}_{0}$. By Lemma \ref{lem3}, $P(H(\{y_{1}\},\{\theta\}))>0$, where $H(\{y_{1}\},\{\theta\})$ denotes the hereditarily disjoint stratification of $\{y_{1}\}$ and $\{\theta\}$. Let $B_{1}=H(\{y_{1}\},\{\theta\})$, if $\beta y_{1}=\theta$ for some $\beta \in\tilde{I}_{B_{1}}L^{0}(\mathcal{F},K)$ then $I_{[|\beta|>0]} y_{1}=\beta^{-1}\beta y_{1}=\theta$. By the choice of $B_{1}$ it follows that $I_{[|\beta|>0]}=0$, which implies $\beta=0$. Thus $\{\tilde{I}_{B_{1}}y_{1}\}$ is $\tilde{I}_{B_{1}}L^{0}(\mathcal{F},K)$-independent.

Suppose for some $k\in N$ such that $1\leqslant k \leqslant n$ there exist an $\mathcal{F}$-measurable subset $B_{k}$ of $B$ and $\{y_{1},y_{2},\cdots,y_{k}\}\subset E$ such that $P(B_{k})>0$ and $\{\tilde{I}_{B_{k}}y_{i}~|~1\leqslant i\leqslant k\}$ is $\tilde{I}_{B_{k}}L^{0}(\mathcal{F},K)$-independent, further let $M:=\{\sum_{i=1}^{k}\xi_{i}y_{i}~|~\xi_{i} \in \tilde{I}_{B_{k}}L^{0}(\mathcal{F},K), 1\leqslant i\leqslant k\}$. Since $B_{k}\notin \mathcal{A}_{k}$, there exists $y_{k+1}\in (\tilde{I}_{B_{k}}E)\setminus M$; further let $B_{k+1}=H(M,\{y_{k+1}\})$, then it is also clear that $B_{k+1}\subset B_{k}$ and $P(B_{k+1})>0$. If $\sum_{i=1}^{k+1}\beta_{i}y_{i}=\theta$ for some $\beta_{i} \in\tilde{I}_{B_{k+1}}L^{0}(\mathcal{F},K)$, then
$I_{[|\beta_{k+1}|>0]}y_{k+1}=\sum_{i=1}^{k}(-\beta_{k+1}^{-1}\beta_{i})y_{i}$. By the choice of $B_{k+1}$ it follows that $I_{[|\beta_{k+1}|>0]}=0$, which implies $\beta_{k+1}=0$, and hence also $\beta_{i}=0$ for $1\leqslant i\leqslant k$ since $\{\tilde{I}_{B_{k}}y_{i}~|~1\leqslant i\leqslant k\}$ is $\tilde{I}_{B_{k}}L^{0}(\mathcal{F},K)$-independent, so that $\{\tilde{I}_{B_{k+1}}y_{i}~|~1\leqslant i\leqslant k+1\}$ is $\tilde{I}_{B_{k+1}}L^{0}(\mathcal{F},K)$-independent.

Consequently, by induction we can obtain an $\mathcal{F}$-measurable set $B_{n+1}$ and $\{y_{1},y_{2},\cdots,y_{n+1}\}\subset E$ such that $P(B_{n+1})>0$ and $\{\tilde{I}_{B_{n+1}}y_{i}~|~1\leqslant i\leqslant n+1\}$ is $\tilde{I}_{B_{n+1}}L^{0}(\mathcal{F},K)$-independent, but this is impossible. Otherwise, suppose $\tilde{I}_{B_{n+1}}y_{i}=\sum_{j=1}^{n}\xi_{ji}x_{j}$ for some $\xi_{ji}\in \tilde{I}_{B_{n+1}}L^{0}(\mathcal{F},K)$, where $1\leqslant j\leqslant n$ and $1\leqslant i\leqslant n+1$. Now consider the system of linear equations
\begin{equation}\left\{
\begin{array}{ccccccccc}
\xi_{11}\lambda_{1}&+&\xi_{12}\lambda_{2}&+&\cdots&+&\xi_{1,n+1}\lambda_{n+1}&=&0\\
\xi_{21}\lambda_{1}&+&\xi_{22}\lambda_{2}&+&\cdots&+&\xi_{2,n+1}\lambda_{n+1}&=&0\\
\multicolumn{9}{c}{\dotfill}\\
\xi_{n1}\lambda_{1}&+&\xi_{n2}\lambda_{2}&+&\cdots&+&\xi_{n,n+1}\lambda_{n+1}&=&0.\\
\end{array}
\right
.\end{equation}
By Lemma \ref{lem1} there exists a nontrivial solution $(\lambda_{1},\lambda_{2},\cdots,\lambda_{n+1})\in \tilde{I}_{B_{n+1}}L^{0}(\mathcal{F},K^{n+1})$ satisfying (2.2). It is easy to check that $\sum_{i=1}^{n+1}\lambda_{i}\tilde{I}_{B_{n+1}}y_{i}=\theta$, which is a contradiction. Thus $P(\Omega \setminus \bigcup_{0\leqslant k \leqslant n}A_{k})=0$.

The desired partition can be obtained easily once we can prove that $P(A_{i}\cap A_{j})=0$
when $0\leqslant i,j \leqslant n$ and $i\neq j$. Suppose $P(A_{i}\cap A_{j})>0$ for some $i,j$ such that $0\leqslant i,j \leqslant n$, then it is easy to verify
that $\tilde{I}_{A_{i}\cap A_{j}}E$ is a free module of both rank $i$ and $j$ over the algebra
$\tilde{I}_{A_{i}\cap A_{j}}L^{0}(\mathcal{F},K)$. Since $\tilde{I}_{A_{i}\cap A_{j}}L^{0}(\mathcal{F},K)$
is a commutative ring with identity $\tilde{I}_{A_{i}\cap A_{j}}$, it follows from \cite[Chapter 4, Corollary 2.12]{alge} that
$\tilde{I}_{A_{i}\cap A_{j}}E$ has the invariant dimension property, which implies $i=j$ and at the same time also proves the uniqueness of each $A_{i}$.
\end{pot}

In the later part of this section, we want to give two interesting results of finitely generated $L^{0}(\mathcal{F},K)$-modules, which are closely related to the countable concatenation property. First, let us introduce a special finitely generated $L^{0}(\mathcal{F},K)$-module. Denote by $L^{0}(\mathcal{F},K^{n})$ the linear space of equivalence classes of $K^{n}$--valued random variables on $(\Omega,\mathcal{F},P)$, where $n$ is some positive integer, define the module multiplication
$\cdot:L^{0}(\mathcal{F},K)\times L^{0}(\mathcal{F},K^{n})\rightarrow L^{0}(\mathcal{F},K^{n})$ by
$\lambda \cdot x=(\lambda\xi_{1},\lambda\xi_{2},\cdots,\lambda\xi_{n}), \forall\lambda\in L^{0}(\mathcal{F},K)$ and $ x=(\xi_{1},\xi_{2},\cdots,\xi_{n})$ $\in L^{0}(\mathcal{F},K^{n})$,  then clearly $L^{0}(\mathcal{F},K^{n})$ is a free $L^{0}(\mathcal{F},K)$-module of rank $n$.

\begin{Corollary}\label{cor1}
An $L^{0}(\mathcal{F},K)$-module $E$ is finitely generated if and only if $E$ is module isomorphic to a submodule $M$ of $L^{0}(\mathcal{F},K^{n})$ such that $M$ has the countable concatenation property, where $n$ denotes some positive integer.
\end{Corollary}

\begin{proof}
Necessity. By Theorem \ref{thm1} there exist a positive integer $n$ and a finite partition $\{A_{0},A_{1},\cdots,A_{n}\}$ of $\Omega$ to $\mathcal{F}$ such that
$\tilde{I}_{A_{i}}E$ is a free module of rank $i$ over the algebra
$\tilde{I}_{A_{i}}L^{0}(\mathcal{F},K)$ for each $i$ which satisfies $0\leqslant i\leqslant n$ and $P(A_{i})>0$.
Let $L=\{i~|~1\leqslant i\leqslant n$ and $P(A_{i})>0\}$, then for each $i\in L$ there
exists a basis $\{x_{k}^{i}\in \tilde{I}_{A_{i}}E~|~1\leqslant k\leqslant i\}$ for the free $\tilde{I}_{A_{i}}L^{0}(\mathcal{F},K)$-module $\tilde{I}_{A_{i}}E$. It follows that for any $x\in E$ there uniquely exists a set $\{\xi_{k}^{i}\in \tilde{I}_{A_{i}}L^{0}(\mathcal{F},K)~|~1\leqslant k\leqslant i\}$ for any $i\in L$ such that $x=\sum_{i\in L}\sum_{1\leqslant k\leqslant i}\xi_{k}^{i}x_{k}^{i}$. Let
$\xi_{k}=\sum_{i\in L, i\geqslant k}\xi_{k}^{i}$ for each $k$ such that $1\leqslant k\leqslant n$ and $\{i\in L~|~i\geqslant k\}\neq \emptyset$, and $\xi_{k}=0$ otherwise, then we define $T: E\rightarrow L^{0}(\mathcal{F},K^{n})$ by $T(x)=(\xi_{1},\xi_{2},\cdots,\xi_{n})$. It is easy to check
that $T$ is an injective module homomorphism and $T(E)$ has the countable concatenation property since $T(E)$ is also
finitely generated.

Sufficiency. Suppose $M$ is a submodule of $L^{0}(\mathcal{F},K^{n})$ such that $M$ has the countable concatenation property.
For each nonnegative integer $k$, let us define$$
\mathcal{B}_{k}=\{B\in \mathcal {F}~|~\tilde{I}_{B}M~\mbox{is a free module of rank}~k~\mbox{over the algebra}~\tilde{I}_{B}L^{0}(\mathcal{F},K)\},$$
then it is easy to check that $\mathcal{B}_{k}=\emptyset$ for $k>n$. Let $B_{k}=$esssup$(\mathcal{B}_{k})$ for $0\leqslant k\leqslant n$ and $\mathcal{B}_{k}\neq \emptyset$, then
$B_{k}\in \mathcal{B}_{k}$ follows from the fact that $M$ has the countable concatenation property. Besides, $P(B_{j}\cap B_{k})=0$ if
$j\neq k$ and $\mathcal{B}_{j}$ and $\mathcal{B}_{k}$ are not empty. If $P(\bigcup\{B_{k}~|~0\leqslant k\leqslant n~\mbox{and}~ \mathcal{B}_{k}\neq \emptyset\})<1$, then by the argument used in the proof of Theorem \ref{thm1} there exist an $\mathcal{F}$-measurable set $A\subset \Omega\setminus \bigcup\{B_{k}~|~0\leqslant k\leqslant n~\mbox{and}~ \mathcal{B}_{k}\neq \emptyset\}$ and $\{x_{1},x_{2},\cdots,x_{n+1}\}\subset M$ such that $P(A)>0$ and $\{\tilde{I}_{A}x_{i}~|~1\leqslant i\leqslant n+1\}$ is $\tilde{I}_{A}L^{0}(\mathcal{F},K)$-independent, which is a contradiction. Thus $M$ is a direct sum of finite many finitely generated $L^{0}(\mathcal{F},K)$-modules, which implies $M$ is also finitely generated.
\end{proof}

\begin{remark}\label{rek}
There exist submodules of $L^{0}(\mathcal{F},K^{n})$ which do not have the countable concatenation property  and thus are not finitely generated. For example, let $\Omega=[0,1]$, $\mathcal{F}=$ the collection of Lebesgue measurable subsets of $[0,1]$ and $P=$ the Lebesgue measure on $[0,1]$. Suppose $M=\{\tilde{I}_{[2^{-n},2^{-n+1}]}~|~n\in N\}$ and $E=\{\sum_{i=1}^{n}\xi_{i}x_{i}~|~\xi_{i}\in L^{0}(\mathcal{F},K),x_{i}\in M,1\leqslant i\leqslant n$ and $n\in N\}$, then it is easy to see that $E$ is a submodule of
$L^{0}(\mathcal{F},K)$ such that $E$ does not have the countable concatenation property and is not a finitely generated $L^{0}(\mathcal{F},K)$-module.
\end{remark}

Similar to the notion of an $RN$ module, we have the notion of a random inner product module (briefly, an $RIP$ module) over $K$ with base $(\Omega,\mathcal{F},P)$ (see \cite{base,rela} for details). Define $\langle\cdot,\cdot\rangle:L^{0}(\mathcal{F},K^{n})\times L^{0}(\mathcal{F},K^{n})\rightarrow L^{0}(\mathcal{F},K)$
by $\langle x,y\rangle=\Sigma _{i=1}^{n}\xi_{i}\bar{\eta}_{i}$, $\forall\lambda\in L^{0}(\mathcal{F},K)$, $ x=(\xi_{1},\xi_{2},\cdots,\xi_{n})$ and $y=(\eta_{1},\eta_{2},\cdots,\eta_{n})\in L^{0}(\mathcal{F},K^{n})$.
It is easy to check that $(L^{0}(\mathcal{F},K^{n}),\langle\cdot,\cdot\rangle)$ is an $RIP$ module over $K$ with base $(\Omega,\mathcal{F},P)$ and it is, of course, also an $RN$ module. Moreover, $(L^{0}(\mathcal{F},K^{n}),\langle\cdot,\cdot\rangle)$ is complete with respect to the topology of convergence in probability $P$, which is exactly the $(\varepsilon,\lambda)$-topology on $(L^{0}(\mathcal{F},K^{n}),\langle\cdot,\cdot\rangle)$ (see Section 3). Specially, $L^{0}(\mathcal{F},K)$ is an $RN$ module and $\|\lambda\|=|\lambda|$ for any $\lambda\in L^{0}(\mathcal{F},K)$.

Recall that if $X$ is a proper linear subspace of $K^{n}$, then there exists
$x\in K^{n}$ such that $x\neq 0$ and $(x,y)=0$,$\forall y \in X$, where $(\cdot,\cdot)$ denotes the usual inner product. Corollary \ref{cor2} below shows that a proper submodule of $L^{0}(\mathcal{F},K^{n})$ with the countable concatenation property has a similar property.

\begin{Corollary}\label{cor2}
Suppose $M$ is a proper submodule of $L^{0}(\mathcal{F},K^{n})$ such that $M$ has the countable concatenation property, then there exists $x\in L^{0}(\mathcal{F},K^{n})$ such that
$x\neq 0$ and $\langle x,y\rangle=0$, $\forall y\in M$.
\end{Corollary}

\begin{proof}
Notice that $M$ is a proper submodule of
$L^{0}(\mathcal{F},K^{n})$ and has the countable concatenation property, it follows from Theorem \ref{thm1} and Corollary \ref{cor1}  there exists a partition $\{A_{0},A_{1},\cdots,A_{n}\}$ of $\Omega$ to $\mathcal{F}$ such that $M$ is free on $A_{i}$ for $0\leqslant i\leqslant n$. By \cite[Lemma 3.7]{acha} each $\tilde{I}_{A_{i}}M$ is a closed submodule of $L^{0}(\mathcal{F},K^{n})$ with respect to the $(\varepsilon,\lambda)$-topology, and it follows that $M$ is also a closed submodule of $L^{0}(\mathcal{F},K^{n})$ since $M=\bigoplus_{i=0}^{n}\tilde{I}_{A_{i}}M$. Consequently the existence of the desired element follows from the orthogonal decomposition theorem \cite[Corollary 4.1]{base}.
\end{proof}

\section{Helly Theorem}
\label{}

The proof of Theorem \ref{thm2} needs Lemma \ref{lem6} below as well as Theorem \ref{thm1}. To introduce Lemma \ref{lem6}, we give the notion of a random locally convex module, which includes the notion of an $RN$ module as a special case as follows:

\begin{definition}[\cite{pro,rela}]
An ordered pair $(E,\mathcal{P})$ is called a random locally convex module over $K$ with base $(\Omega,\mathcal{F},P)$ if $E$ is a left module over $L^{0}(\mathcal{F},K)$ and $\mathcal{P}$ is a family of $L^{0}$-seminorms such that $\bigvee\{\|x\|~|~\|\cdot\|\in \mathcal{P}\}=0$ if and only if $x=\theta$.
\end{definition}

Clearly, when $\mathcal{P}$ reduces to a singleton $\{\|\cdot\|\}$, then a random locally convex module $(E,\mathcal{P})$ is exactly an $RN$ module.

Given a random locally convex module $(E,\mathcal{P})$ over $K$ with base $(\Omega,\mathcal{F},P)$, we always denote the set of finite subfamilies of $\mathcal{P}$ by $\mathcal{F}(\mathcal{P})$. For each $\mathcal{Q}\in \mathcal{F}(\mathcal{P})$, define $\|\cdot\|_{\mathcal{Q}}:E\rightarrow L^{0}_{+}(\mathcal{F})$ by  $\|x\|_{\mathcal{Q}}=\bigvee\{\|x\|~|~\|\cdot\|\in \mathcal{Q}\},\forall x\in E$. Let $\varepsilon$ and $\lambda$ be any two positive numbers such that $0<\lambda <1$, define $N_{\theta}(\mathcal{Q},\varepsilon,\lambda)=\{x\in E~|~P(\{\omega\in \Omega~|~\|x\|_{\mathcal{Q}}(\omega)<\varepsilon\})>1-\lambda\}$ and denote $\mathcal{N}_{\theta}=\{N_{\theta}(\mathcal{Q},\varepsilon,\lambda)~|~\mathcal{Q}\in \mathcal{F}(\mathcal{P}), \varepsilon>0, 0<\lambda <1\}$. Then $\mathcal{N}_{\theta}$ becomes a local base at $\theta$ of some Hausdorff linear topology, called the $(\varepsilon,\lambda)$-topology for $(E,\mathcal{P})$. As shown in \cite{pro,rela}, $L^{0}(\mathcal{F},K)$ is a topological algebra over $K$ in the $(\varepsilon,\lambda)$-topology, and a random locally convex module $(E,\mathcal{P})$ over $K$ with base $(\Omega,\mathcal{F},P)$ is a topological module over the topological algebra $L^{0}(\mathcal{F},K)$ when $E$ and $L^{0}(\mathcal{F},K)$ are endowed with their $(\varepsilon,\lambda)$-topologies, respectively. In the sequel of this paper, we always denote by $\mathcal{T}_{\varepsilon,\lambda}$ the $(\varepsilon,\lambda)$-topology for every random locally convex module whenever no confusion exists.

Filipovi\'{c}, Kupper and Vogelpoth \cite{DMN} introduced a new topology for a random locally convex module: given a random locally convex module $(E,\mathcal{P})$ over $K$ with base $(\Omega,\mathcal{F},P)$, a subset $G$ of $E$ is $\mathcal{T}_{c}$-open if for each $x\in G$ there exist a finite subfamily $\mathcal{Q}$ of $\mathcal{P}$ and $\epsilon\in L^{0}_{++}(\mathcal{F})$ such that $x+B_{\mathcal{Q}}(\epsilon)\subset G$, where $B_{\mathcal{Q}}(\epsilon)=\{y\in E~|~\|y\|_{\mathcal{Q}}\leqslant \epsilon\}$. Denote $\mathcal{T}_{c}$ by the family of $\mathcal{T}_{c}$-open subsets of $E$, then $\mathcal{T}_{c}$ becomes a Hausdorff topology and $(E,\mathcal{T}_{c})$ is a locally $L^{0}$-convex module in the sense of \cite{DMN}, so we often call the topology $\mathcal{T}_{c}$ the locally $L^{0}$-convex topology induced by $\mathcal{P}$. As shown in \cite{DMN}, $L^{0}(\mathcal{F},K)$ is a topological ring and $(E,\mathcal{T}_{c})$ is a topological module over the topological ring $L^{0}(\mathcal{F},K)$ when $(E,\mathcal{P})$ and $L^{0}(\mathcal{F},K)$ are endowed with their locally $L^{0}$-convex topologies, respectively. From now on, we always denote by $\mathcal{T}_{c}$ the locally $L^{0}$-convex topology for each random locally convex module whenever no confusion exists.

Given a random locally convex module $(E,\mathcal{P})$ over $K$ with base $(\Omega,\mathcal{F},P)$, let $E^{*}_{\varepsilon,\lambda}$ be the $L^{0}(\mathcal{F},K)$-module of continuous module homomorphisms from $(E,\mathcal{T}_{\varepsilon,\lambda})$ to $(L^{0}(\mathcal{F},K),\mathcal{T}_{\varepsilon,\lambda})$, called the random conjugate space of $(E,\mathcal{P})$ under the $(\varepsilon,\lambda)$-topology, and $E^{*}_{c}$ the $L^{0}(\mathcal{F},K)$-module of continuous module homomorphisms from $(E,\mathcal{T}_{c})$ to $(L^{0}(\mathcal{F},K),\mathcal{T}_{c})$, called the random conjugate space of $(E,\mathcal{P})$ under the locally $L^{0}$-convex topology. It was proved in \cite{rela} that $E^{*}_{c}\subset E^{*}_{\varepsilon,\lambda}$ and that $E^{*}_{c}=E^{*}_{\varepsilon,\lambda}$ if $\mathcal{P}$ has the countable concatenation property, where $\mathcal{P}$ is said to have the countable concatenation property \cite{DMN} if $\sum_{n\in N}\tilde{I}_{A_{n}}\|\cdot\|_{\mathcal{Q}_{n}}\in \mathcal{P}$ for any countable partition $\{A_{n}~|~n\in N\}$ of $\Omega$ to $\mathcal{F}$ and any sequence $\{\mathcal{Q}_{n},n\in N\}$ of finite subfamilies of $\mathcal{P}$. Specially, we have $E^{*}_{c}=E^{*}_{\varepsilon,\lambda}$ for every $RN$ module $(E,\|\cdot\|)$, so we always use $E^{*}$ for $E^{*}_{c}$ or $E^{*}_{\varepsilon,\lambda}$ for an $RN$ module
$(E,\|\cdot\|)$.

In fact, we proved in \cite{exte,rela,james} that a linear operator $f$ from an $RN$ module $(E,\|\cdot\|)$ over $K$ with base $(\Omega,\mathcal{F},P)$ to $L^{0}(\mathcal{F},K)$ belongs to $E^{*}$ if and only if there exists some $\xi$ in $L^{0}_{+}(\mathcal{F})$ such that $|f(x)|\leqslant \xi \|x\|$, $\forall x \in E$, thus $E^{*}$ coincides with Definition 1.3 and $\|f\|^{*}=\bigvee\{|f(x)|~|~x\in E$ and $\|x\|\leqslant1\}$.

\begin{lemma}\label{lem4}
Let $(E,\mathcal{P})$ be a random locally convex module over $K$ with base $(\Omega,\mathcal{F},P)$ such that $\mathcal{P}$ has the countable concatenation property. If a subset $G$ of $E$ has the countable concatenation property, then so does the $\mathcal{T}_{c}$-interior $G^{\circ}$ of $G$.
\end{lemma}

\begin{proof}
Suppose $G^{\circ}\neq \emptyset$, $\{x_{n}~|~n\in N\}\subset G^{\circ}$ and $\{A_{n}~|~n\in N\}$ is a countable partition of $\Omega$ to $\mathcal{F}$, then there exists $x\in G$ such that $x=\Sigma_{n\in N}\tilde{I}_{A_{n}}x_{n}$ by the countable concatenation property of $G$. If a sequence $\{\mathcal{Q}_{n},n\in N\}$ of finite subfamilies of $\mathcal{P}$ and $\{\epsilon_{n}\in L^{0}_{++}(\mathcal{F})~|~n\in N\}$ satisfy $x_{n}+B_{\mathcal{Q}_{n}}(\epsilon_{n})\subset G$ for each $n\in N$, then it is easy to check that $x+B_{\{\|\cdot\|\}}(\epsilon)\subset G$, where $\|\cdot\|=\sum_{n\in N}\tilde{I}_{A_{n}}\|\cdot\|_{\mathcal{Q}_{n}}$ and $\epsilon=\Sigma_{n\in N}\tilde{I}_{A_{n}}\epsilon_{n}$, so that $G^{\circ}$ has the countable concatenation property.
\end{proof}

In fact, Lemma \ref{lem4} motivates an interesting result, which is Proposition \ref{pro3} below.
\begin{Proposition}\label{pro3}
If a random locally convex module $(E,\mathcal{P})$ possesses a nonempty $\mathcal{T}_{c}$-open subset $G$ such that $G$ has the countable concatenation property, then $E$ must have the countable concatenation property.
\end{Proposition}

\begin{proof}
We can, without loss of generality, suppose $\theta\in G$. Let $\{y_{n}~|~n\in N\}$ be a sequence in $E$ and $\{A_{n}~|~n\in N\}$ a countable partition of $\Omega$ to $\mathcal{F}$, then for each $n\in N$ there exist a finite subset $\mathcal{Q}_{n}$ of $\mathcal{P}$ and $\lambda_{n}\in L^{0}_{++}(\mathcal{F})$ such that
$$\{\xi x~|~\xi\in L^{0}(\mathcal{F},K)~\mbox{such that}~|\xi|\leqslant 2\lambda_{n}~\mbox{and}~x\in y_{n}+B_{\mathcal{Q}_{n}}(\lambda_{n})\}\subset G$$
since $0\cdot y_{n}=\theta$ and the module multiplication is continuous with respect to $\mathcal{T}_{c}$,
which also implies $\lambda_{n}y_{n}\in G$. Moreover, by the countable concatenation property of $G$ there exists $y\in G$ such that $y=\Sigma_{n\in N}\tilde{I}_{A_{n}}\lambda_{n}y_{n}$, then letting $\eta=\Sigma_{n\in N}\tilde{I}_{A_{n}}\lambda_{n}^{-1}$ one can easily check $\eta y=\Sigma_{n\in N}\tilde{I}_{A_{n}}y_{n}$, namely $\Sigma_{n\in N}\tilde{I}_{A_{n}}y_{n}\in E$.
\end{proof}

Let us recall that $G$ is an $L^{0}$-convex subset of an $L^{0}(\mathcal{F},K)$-module $E$ if $\lambda x+(1-\lambda)y\in G$ for all $x,y\in G$ and $\lambda\in L^{0}_{+}(\mathcal{F})$ such that $\lambda\leqslant 1$. Lemma \ref{lem6} below is merely a variant of \cite[Theorem 3.15]{rela}.

\begin{lemma}\label{lem6}
Let $(E,\mathcal{P})$ be a random locally convex module over $K$ with base $(\Omega,\mathcal{F},P)$ and $G$ and $M$ two nonempty $L^{0}$-convex subsets of $E$ such that the
$\mathcal{T}_{c}$-interior $G^{\circ}$ of $G$ is not empty and $H_{cc}(G^{\circ})\cap H_{cc}(M)=\emptyset$. Then there exists $f \in E^{*}_{c}$ such that

$(Ref)(x)\leqslant (Ref)(y)$ for all $x \in G$ and $y\in M$

\noindent and

$(Ref)(x)< (Ref)(y)$ on $H(G^{\circ},M)$ for all $x \in G^{\circ}$ and $y\in M$.

\noindent Where $(Ref)(x)=Re(f(x))$, $\forall x\in E$.
\end{lemma}

\begin{proof}
Clearly $G^{\circ}$ is also $L^{0}$-convex, then it follows from \cite[Theorem 3.15]{rela} that there exists $f \in E^{*}_{c}$ such that

$(Ref)(x)<(Ref)(y)$ for all $x \in G^{\circ}$ and $y\in M$

\noindent and

$(Ref)(x)< (Ref)(y)$ on $H(G^{\circ},M)$ for all $x \in G^{\circ}$ and $y\in M$.

\noindent It follows that $(Ref)(x)\leqslant (Ref)(y)$ for all $x \in G$ and $y\in M$ since $f$ is continuous and $G$ is included in the $\mathcal{T}_{c}$-closure of $G^{\circ}$.
\end{proof}

In this paper we only need the following special case of Lemma \ref{lem6}:

\begin{Corollary}\label{lem5}
Let $(E,\mathcal{P})$ be a random locally convex module over $K$ with base $(\Omega,\mathcal{F},P)$ such that $\mathcal{P}$ has the countable concatenation property and $G$ and $M$ two nonempty $L^{0}$-convex subsets of $E$ such that  $G$ and $M$ have the countable concatenation property, the
$\mathcal{T}_{c}$-interior $G^{\circ}$ of $G$ is not empty and $G^{\circ}\cap M=\emptyset$. Then there exists $f \in E^{*}_{c}$ such that

$(Ref)(x)\leqslant (Ref)(y)$ for all $x \in G$ and $y\in M$

\noindent and

$(Ref)(x)< (Ref)(y)$ on $H(G^{\circ},M)$ for all $x \in G^{\circ}$ and $y\in M$.

\end{Corollary}

\begin{proof}
By Lemma \ref{lem4} $G^{\circ}$ has the countable concatenation property, and hence $H_{cc}(G^{\circ})=G^{\circ}$ and $H_{cc}(M)=M$, our desired result follows from Lemma \ref{lem6}.
\end{proof}

We can now prove Theorem \ref{thm2}.

\newproof{pot1}{Proof of Theorem \ref{thm2}}
\begin{pot1}
Necessity is obvious, it remains to prove sufficiency.

Let $S=\{\Sigma_{i=1}^{n}\zeta_{i}f_{i}~|~\zeta_{i}\in L^{0}(\mathcal{F},K), 1\leqslant i\leqslant n\}$, then $S$ is a finitely generated $L^{0}(\mathcal{F},K)$ module. By Theorem \ref{thm1} there exists a finite partition $\{A_{0},A_{1},\cdots,A_{n}\}$ of $\Omega$ to $\mathcal{F}$ such that $\tilde{I}_{A_{i}}S$ is a free $\tilde{I}_{A_{i}}L^{0}(\mathcal{F},K)$-module of rank $i$ for each $i$ which satisfies $0\leqslant i\leqslant n$ and $P(A_{i})>0$. Let $\{g_{j}\in \tilde{I}_{A_{i}}S~|~1\leqslant j\leqslant i\}$ be a basis for $\tilde{I}_{A_{i}}S$ for some $i$ such that $1\leqslant i\leqslant n$ and $P(A_{i})>0$, and suppose $g_{j}=\Sigma_{k=1}^{n}\zeta_{kj}f_{k}$ for some $\zeta_{kj}\in \tilde{I}_{A_{i}}L^{0}(\mathcal{F},K)$, where $1\leqslant k\leqslant n$ and $1\leqslant j\leqslant i$.

Let $\gamma_{j}=\Sigma_{k=1}^{n}\zeta_{kj}\xi_{k}$($1\leqslant j\leqslant i$), then
$$|\sum_{j=1}^{i}\lambda_{j}\gamma_{j}|=|\sum_{j=1}^{i}\sum_{k=1}^{n}\lambda_{j}\zeta_{kj}\xi_{k}|\leqslant \beta\|\sum_{j=1}^{i}\sum_{k=1}^{n}\lambda_{j}\zeta_{kj}f_{k}\|^{*}=\beta\|\sum_{j=1}^{i}\lambda_{j}g_{j}\|^{*}$$
for any $\lambda_{1},\lambda_{2},\cdots \lambda_{n}\in L^{0}(\mathcal{F},K)$.

If for each $i$ such that $P(A_{i})>0$ there exists $x_{A_{i}}\in \tilde{I}_{A_{i}}E$ such that $\|x_{A_{i}}\|\leq \tilde{I}_{A_{i}}(\beta+\varepsilon)$ and $g_{j}(x_{A_{i}})=\gamma_{j}$ for each $j$ such that $1\leqslant j\leqslant i$, then $f_{k}(x_{A_{i}})=\tilde{I}_{A_{i}}\xi_{k}$ for each $k$ such that $1\leqslant k\leqslant n$. In fact, suppose $\tilde{I}_{A_{i}}f_{k}=\sum_{j=1}^{i}\eta_{jk}g_{j}$ for some $\eta_{jk}\in \tilde{I}_{A_{i}}L^{0}(\mathcal{F},K)$ ($1\leqslant j\leqslant i$ and $1\leqslant k\leqslant n$), then
$$|\tilde{I}_{A_{i}}\xi_{k}-\sum_{j=1}^{i}\eta_{jk}\gamma_{j}|\leqslant\beta\|\tilde{I}_{A_{i}}f_{k}-\sum_{j=1}^{i}
\eta_{jk}g_{j}\|^{*}=0,$$
i.e. $\tilde{I}_{A_{i}}\xi_{k}=\sum_{j=1}^{i}\eta_{jk}\gamma_{j}$ ($1\leqslant k\leqslant n$). Hence
$$f_{k}(x_{A_{i}})=\sum_{j=1}^{i}\eta_{jk}g_{j}(x_{A_{i}})=\sum_{j=1}^{i}\eta_{jk}\gamma_{j}=\tilde{I}_{A_{i}}\xi_{k},
~k=1,2, \cdots,n.$$

If $P(A_{i})=0$, then we define $x_{A_{i}}=0$. Finally, we define $x_{\varepsilon}=\sum_{i=0}^{n}x_{A_{i}}$, then $x_{\varepsilon}$ will satisfy (1) $f_{i}(x_{\varepsilon})=\xi_{i}$ for any $i$ such that $1\leqslant i\leqslant n$ and (2) $\|x_{\varepsilon}\|\leqslant \beta+\varepsilon$.

Thus we can, without loss of generality, suppose
$\{f_{1},f_{2},\cdots,f_{n}\}$ is $L^{0}(\mathcal{F},K)$-independent, otherwise we can consider $\tilde{I}_{A_{i}}E$ for each $i$ such that $P(A_{i})>0$, further take $\xi_{k}^{i}=\tilde{I}_{A_{i}}\xi_{k}$, $f_{k}^{i}=\tilde{I}_{A_{i}}f_{k}$ for each $k$ such that $1\leq k \leq n$, $\beta^{i}=\tilde{I}_{A_{i}}\beta$ and $\varepsilon^{i}=\tilde{I}_{A_{i}}\varepsilon$, since $\tilde{I}_{A_{i}}E$ can be regarded as an $RN$ module with base $(A_{i},A_{i}\cap\mathcal{F},P_{i})$, where $P_{i}: A_{i}\cap\mathcal{F}\rightarrow [0,1]$ is defined by $P_{i}(A_{i}\cap A)=P(A_{i}\cap A) / P(A_{i})$ for any $A\in \mathcal{F}$, once we prove this theorem for the case when $\{f_{1},f_{2},\cdots,f_{n}\}$ is $L^{0}(\mathcal{F},K)$-independent, we can apply the proved case to each $\tilde{I}_{A_{i}}E$ such that $P(A_{i})>0$.

Now let us define $T:E\rightarrow L^{0}(\mathcal{F},K^{n})$
by $Tx=(f_{1}(x),f_{2}(x),\cdots ,f_{n}(x))$, $\forall x \in E$,
then it is obvious that $T(E)$ is a submodule of $L^{0}(\mathcal{F},K^{n})$ and $T(E)$ has the countable concatenation property.
If $T(E)\neq L^{0}(\mathcal{F},K^{n})$, by Corollary \ref{cor2} there exists a nontrivial element
$z=(\eta_{1},\eta_{2},\cdots,\eta_{n})\in L^{0}(\mathcal{F},K^{n})$ such that
$$
(\sum_{k=1}^{n} \bar{\eta}_{k}f_{k})(x)=\sum_{k=1}^{n} \bar{\eta}_{k}f_{k}(x)=\langle T(x),z\rangle=0, \forall  x\in E,
$$
but this contradicts with the $L^{0}(\mathcal{F},K)$-independence of $\{f_{1},f_{2},\cdots,f_{n}\}$, and consequently $T(E)=L^{0}(\mathcal{F},K^{n})$.

Suppose $x_{1},x_{2},\cdots,x_{n}\in E$ such that $Tx_{i}=(\eta_{1}^{i},\eta_{2}^{i},\cdots,\eta_{n}^{i})$,
$\eta_{i}^{i}=1$ and $\eta_{j}^{i}=0(i\neq j)$ ($1\leqslant j \leqslant n$ and $1\leqslant i \leqslant n$).
Let $\gamma=\bigvee_{i=1}^{n}\|x_{i}\|$, clearly $\gamma>0$ on $\Omega$. If
$y=(\alpha_{1},\alpha_{2},\cdots,\alpha_{n}) \in L^{0}(\mathcal{F},K^{n})$ and
$\|y\|\leqslant(\beta+\varepsilon)n^{-1}\gamma^{-1}$ for some fixed $\varepsilon\in L^{0}_{++}(\mathcal{F})$,
then $T(\sum_{i=1}^{n} \alpha_{i}x_{i})=y$ and $$
\|\sum_{i=1}^{n} \alpha_{i}x_{i}\| \leqslant \sum_{i=1}^{n} |\alpha_{i}|\|x_{i}\|\leqslant(\beta+\varepsilon)n^{-1}\gamma^{-1}\sum_{i=1}^{n}\|x_{i}\|\leqslant\beta+\varepsilon.$$
Let $\bar{B}_{\beta+\varepsilon}=\{x\in E|\;\|x\|\leqslant\beta+\varepsilon\}$, the above argument shows that
$T(\bar{B}_{\beta+\varepsilon})$ contains an $\mathcal{T}_{c}$-open neighborhood
$\{y\in L^{0}(\mathcal{F},K^{n})~|~ \|y\|\leqslant(\beta+\varepsilon)n^{-1}\gamma^{-1}\}$ of the null
element of $L^{0}(\mathcal{F},K^{n})$. Moreover, it is easy to see that $T(\bar{B}_{\beta+\varepsilon})$
is also an $L^{0}$-convex subset with the countable concatenation property.

If $p:=(\xi_{1}, \xi_{2},\cdots ,\xi_{n})\notin T(\bar{B}_{\beta+\varepsilon})$ for some $\varepsilon\in L^{0}_{++}(\mathcal{F})$, then by Corollary \ref{lem5} there exists $f\in L^{0}(\mathcal{F},K^{n})^{*}$ such that
$(Ref)(y)\leqslant (Ref)(p)$ for all
$y \in T(\bar{B}_{\beta+\varepsilon})$, and
$(Ref)(p)>(Ref)(y)$ on $H(\{p\},[T(\bar{B}_{\beta+\varepsilon})]^{\circ})$ for all $y \in [T(\bar{B}_{\beta+\varepsilon})]^{\circ}$, specially, $(Ref)(p)>(Ref)(0)=0$ on $H(\{p\},[T(\bar{B}_{\beta+\varepsilon})]^{\circ})$.
Let $\xi=|f(y)|(f(y))^{-1}$ for any fixed $y\in T(\bar{B}_{\beta+\varepsilon})$,
then $\xi y\in T(\bar{B}_{\beta+\varepsilon})$ and
$$|f(y)|=f(\xi y)=(Ref)(\xi y)\leqslant (Ref)(p)\leqslant |f(p)|.$$

By Riesz's representation theorem in $\mathcal{T}_{c}$-complete $RIP$ module
\cite[Theorem 4.3]{rela} there exists $y_{0}=(\lambda_{1},\lambda_{2},\cdots \lambda_{n})\in L^{0}(\mathcal{F},K^{n})$
such that $f(y)=\langle y,y_{0}\rangle$, $\forall y\in L^{0}(\mathcal{F},K^{n})$. Then
$$
|\sum_{k=1}^{n} \bar{\lambda}_{k}f_{k}(x)|=|f(Tx)|\leqslant |f(p)|=|\sum_{k=1}^{n} \bar{\lambda}_{k}\xi_{k}|,\forall x\in \bar{B}_{\beta+\varepsilon}.
$$
Thus
\begin{center}
$
(\beta+\varepsilon)\|\sum_{k=1}^{n} \bar{\lambda}_{k}f_{k}\|^{*}=\bigvee_{x\in \bar{B}_{\beta+\varepsilon}}
|\sum_{k=1}^{n} \bar{\lambda}_{k}f_{k}(x)|\leqslant |\sum_{k=1}^{n} \bar{\lambda}_{k}\xi_{k}|.$
\end{center}

Since $|f(p)|=|\langle p,y_{0}\rangle |>0$ on $H(\{p\},[T(\bar{B}_{\beta+\varepsilon})]^{\circ})$, it follows that
$\|y_{0}\|>0$ on $H(\{p\},[T(\bar{B}_{\beta+\varepsilon})]^{\circ})$, and thus $\|\sum_{k=1}^{n} \bar{\lambda}_{k}f_{k}\|\neq 0$
on $H(\{p\},[T(\bar{B}_{\beta+\varepsilon})]^{\circ})$ by the $L^{0}(\mathcal{F},K)$-independence of $\{f_{1},f_{2},\cdots,f_{n}\}$, which yields
$$
\beta\|\sum_{k=1}^{n} \bar{\lambda}_{k}f_{k}\|^{*}<|\sum_{k=1}^{n} \bar{\lambda}_{k}\xi_{k}|
$$
\noindent on $H(\{p\},[T(\bar{B}_{\beta+\varepsilon})]^{\circ})$, and in turn a contradiction to the assumption.
\end{pot1}

\begin{remark}It is necessary to require $E$ to have the countable concatenation property in Theorem \ref{thm2}, otherwise
the result may not hold. Here is an example, let $E$ be defined as in Remark \ref{rek} and define $\|\cdot\|:E\rightarrow L^{0}_{+}(\mathcal{F})$ by $\|\eta\|=|\eta|$, $\forall\eta\in E$, then
$(E,\|\cdot\|)$ is also an $RN$ module over $K$ with base $(\Omega,\mathcal{F},P)$. If $f \in E^{*}$ is defined
by $f(\eta)=\eta$, $\forall\eta \in E$ and take $\xi=\tilde{I}_{\Omega}$, $\beta=\tilde{I}_{\Omega}$, then clearly $|\lambda\xi|\leqslant \beta\|\lambda f\|$ (in fact, $|\lambda\xi|= \beta\|\lambda f\|$), $\forall\lambda \in L^{0}(\mathcal{F},K)$, but there does not exist
any $x\in E$ such that $f(x)=\xi$.

\end{remark}

\begin{Corollary}
Suppose $(E,\|\cdot\|)$ is an $RN$ module over $K$ with base $(\Omega,\mathcal{F},P)$ such that $E$ has
the countable concatenation property, then for any $F\in E^{**}$, $\epsilon \in L^{0}_{++}(\mathcal{F})$ and $f_{1},f_{2},\cdots,f_{n} \in E^{*}$  there exists $x_{\epsilon}\in E$ such that

\begin{enumerate}

\item $f_{i}(x_{\epsilon})=F(f _{i})$, $i=1,2,\cdots ,n$;

\item $\|x_{\epsilon}\|\leqslant \|F\|^{**}+\epsilon$.

\end{enumerate}
Where $(E^{**},\|\cdot\|^{**})$ denotes the random conjugate space of $(E^{*},\|\cdot\|^{*})$.

\end{Corollary}

\begin{proof}
Note that $|\sum_{k=1}^{n}\lambda_{k}F(f _{i})|\leqslant \|F\|^{**}\|\sum_{k=1}^{n} \lambda_{k}f_{k}\|^{*}$ holds for any $\lambda_{1},\lambda_{2},\cdots \lambda_{n}\in L^{0}(\mathcal{F},K)$, then the result follows from Theorem \ref{thm2}.
\end{proof}

As another application of Theorem \ref{thm2}, we will give Corollary \ref{cor3} below, let us first recall the notions of random variables with values in a normed space and random linear functionals. In this section $(B,\|\cdot\|)$ always denotes a normed space over $K$.

A mapping $V$ from $(\Omega,\mathcal{F},P)$ to $(B,\|\cdot\|)$ is called a $B$-valued $\mathcal{F}$-random element \cite{BR,hans} if $V^{-1}(G):=\{\omega\in \Omega~|~V(\omega)\in G\}\in \mathcal{F}$ for any open subset $G$ of $B$.
A $B$-valued $\mathcal{F}$-random element is called simple if it only takes finitely many values. Further, a mapping $V$ from $(\Omega,\mathcal{F},P)$ to $(B,\|\cdot\|)$ is called a $B$-valued $\mathcal{F}$-random variable if there exits a sequence $\{V_{n},n\in N\}$ of simple $B$-valued $\mathcal{F}$-random elements such that $\{\|V_{n}(\omega)-V(\omega)\|,n\in N\}$ converges to $0$ as $n$ tends to $\infty$ for each $\omega \in \Omega$. It is easy to see from \cite{hans} that a $B$-valued $\mathcal{F}$-random element is a $B$-valued $\mathcal{F}$-random variable iff its range is a separable subset of $B$, and the notions of a $B$-valued $\mathcal{F}$-random element and a $B$-valued $\mathcal{F}$-random variable coincide when $B$ is separable.

Denote by $\mathcal{L}^{0}(\mathcal{F},B)$ the linear space of $B$-valued $\mathcal{F}$-random variables under the ordinary pointwise addition and scalar multiplication operations, and by $L^{0}(\mathcal{F},B)$ the linear space of equivalence classes of elements in $\mathcal{L}^{0}(\mathcal{F},B)$ under the ordinary operations on equivalence classes, where two elements in $\mathcal{L}^{0}(\mathcal{F},B)$ are called equivalent if they are equal a.s.. In particular, $L^{0}(\mathcal{F},B)$ is exactly $L^{0}(\mathcal{F},K)$ when $B=K$.

$L^{0}(\mathcal{F},B)$ becomes an $L^{0}(\mathcal{F},K)$-module when the module multiplication: $L^{0}(\mathcal{F},K)\times L^{0}(\mathcal{F},B)\rightarrow L^{0}(\mathcal{F},B)$ is defined by $\xi x=$the equivalence class of $\xi^{0} x^{0}$, where $\xi^{0}$ and $x^{0}$ are arbitrarily chosen representatives of $\xi$ in $L^{0}(\mathcal{F},K)$ and $x$ in $L^{0}(\mathcal{F},B)$, respectively, and $(\xi^{0} x^{0})(\omega)=\xi^{0}(\omega)\cdot x^{0}(\omega), \forall \omega \in \Omega$. Further, the norm $\|\cdot\|$ on $B$ induces an $L^{0}$-norm on $L^{0}(\mathcal{F},B)$, still denoted by $\|\cdot\|$, namely $\|x\|=$the equivalence class of $\|x^{0}\|$ for any $x\in L^{0}(\mathcal{F},B)$, where $x^{0}$ is as above and $\|x^{0}\|$ is the composition function of $x^{0}$ and the norm $\|\cdot\|$ on $B$. Then $(L^{0}(\mathcal{F},B),\|\cdot\|)$ becomes an $RN$ module over $K$ with base $(\Omega,\mathcal{F},P)$ and it, clearly, has the countable concatenation property.

A mapping $f^{0}:\Omega\times B\rightarrow K$ is called a random functional if $f^{0}(\cdot,b):\Omega \rightarrow K$ is a $K$-valued $\mathcal{F}$-random variable for each $b\in B$, further if, in addition, $f^{0}(\omega,\cdot):B \rightarrow K$ is a linear (continuous) functional on $B$ for each $\omega \in \Omega$ then $f^{0}$ is called a sample-linear (resp., sample-continuous) random functional.

Given a sample-linear and sample-continuous random functional $f^{0}:\Omega\times B\rightarrow K$, since $r^{0}(\omega):=$sup$\{|f^{0}(\omega,b)|~|~b\in B$ and $\|b\|\leqslant1\}<+\infty$ for each $\omega \in \Omega$, then we can always consider $\|f^{0}\|_{ess.}=$esssup$\{|f^{0}(\cdot,b)|~|~b\in B$ and $\|b\|\leqslant1\}$ as a nonnegative real-valued random variable, called the essential random norm of $f^{0}$. In particular we can take $\|f^{0}\|_{ess.}=r^{0}$ when $B$ is a separable normed space.

Let $f^{0}:\Omega\times B\rightarrow K$ be a sample-linear and sample-continuous random functional, then $f^{0}$ generates an a.s. bounded random linear functional $f:L^{0}(\mathcal{F},B)\rightarrow L^{0}(\mathcal{F},K)$ in the following way: for each $x\in L^{0}(\mathcal{F},B)$, $f(x)$ is defined as the equivalence class of $f^{0}(\cdot,x^{0}(\cdot))$, where $x^{0}$ is an arbitrarily chosen representative of $x$, then it is known from \cite{Radon,rand} that $\|f\|^{*}=$the equivalence class of $\|f^{0}\|_{ess.}$, further by the theory of the lifting propery \cite{ion1,ion2} we proved in \cite{Radon} that when $(\Omega,\mathcal{F},P)$ is a complete probability space every $f\in L^{0}(\mathcal{F},B)^{*}$ can also be generated by a sample-linear and sample-continuous random functional  $f^{0}:\Omega\times B\rightarrow K$ in the way as above, and $f^{0}$ is unique in the sense that $f^{0}(\omega,b)=g^{0}(\omega,b)$ a.s. for each fixed $b\in B$ if $g^{0}: \Omega\times B\rightarrow K$ is also a sample-linear and sample continuous random functional generating $f$.

\begin{Corollary} \label{cor3}
Let $f^{0}_{1},f^{0}_{2},\cdots,f^{0}_{n}:\Omega\times B\rightarrow K$ be $n$'s sample-linear and sample-continuous random functionals, $\xi^{0}_{1},\xi^{0}_{2},\cdots,\xi^{0}_{n}$ any given $n$'s $K$-valued $\mathcal{F}$-random variables and $\beta^{0}:\Omega\rightarrow [0,+\infty)$ an $\mathcal{F}$-random variable. Then for each a.s. positive real-valued $\mathcal{F}$-random variable $\varepsilon^{0}$ there exist a $B$-valued $\mathcal{F}$-random variable $x^{0}_{\varepsilon^{0}}$ and an $\mathcal{F}$-measurable subset $\Omega_{0}$ of probability one such that the following two items are satisfied:

(1) $f^{0}_{i}(\omega,x^{0}_{\varepsilon^{0}}(\omega))=\xi^{0}_{i}(\omega)$ for each $\omega\in \Omega_{0}$ and each $i$ such that $1\leqslant i \leqslant n$;

(2) $\|x^{0}_{\varepsilon^{0}}(\omega)\|\leqslant \beta^{0}(\omega)+\varepsilon^{0}(\omega)$ for each $\omega \in \Omega_{0}$

\noindent iff
\begin{equation}
|\Sigma_{i=1}^{n}\lambda^{0}_{i}(\omega)\xi^{0}_{i}(\omega)|\leqslant \beta^{0}(\omega)\|\Sigma_{i=1}^{n}\lambda^{0}_{i}f^{0}_{i}\|_{ess.}(\omega)~a.s.
\end{equation}
for any given $n$'s $K$-valued $\mathcal{F}$-random variables $\lambda^{0}_{1},\lambda^{0}_{2},\cdots,\lambda^{0}_{n}$.

In particular when $B$ is separable (3.1) is equivalent to the following:

\begin{equation}
|\Sigma_{i=1}^{n}\alpha_{i}\xi^{0}_{i}(\omega)|\leqslant \beta^{0}(\omega)\cdot\mbox{sup}\{|\Sigma_{i=1}^{n}\alpha_{i}f^{0}_{i}(\omega,b)|~|~b\in B~\mbox{and}~\|b\|\leqslant1\}~a.s.
\end{equation}
for any given $\alpha_{1},\alpha_{2},\cdots,\alpha_{n}\in K$.
\end{Corollary}

\begin{proof}
It is clear that (1) and (2) together implies (3.1), and thus we only need to prove that (3.1) also implies (1) and (2). In fact, for each $i$ such that $1\leqslant i \leqslant n$, let $f_{i}$ be the a.s. bounded random linear functional on $L^{0}(\mathcal{F},B)$ generated by $f^{0}_{i}$, $\xi_{i}$ the equivalence class of $\xi^{0}_{i}$, and $\beta$ and $\varepsilon$ the respective equivalence classes of $\beta^{0}$ and $\varepsilon^{0}$, then (3.1) amounts to the following:
$$|\sum_{i=1}^{n}\lambda_{i}\xi_{i}|\leqslant \beta\|\sum_{i=1}^{n}\lambda_{i}f_{i}\|^{*}$$
for any given $\lambda_{1},\lambda_{2},\cdots,\lambda_{n}\in L^{0}(\mathcal{F},K)$.

Thus by Theorem \ref{thm2} there exists $x_{\varepsilon}\in L^{0}(\mathcal{F},B)$ such that $f_{i}(x_{\varepsilon})=\xi_{i}$ for each $i=1,2,\cdots,n$ and $\|x_{\varepsilon}\|\leqslant \beta+\varepsilon$. Let $x^{0}_{\varepsilon^{0}}$ be an arbitrarily chosen representative of $x_{\varepsilon}$, then $f^{0}_{i}(\omega,x^{0}_{\varepsilon^{0}}(\omega))=\xi^{0}_{i}(\omega)$ a.s. for each $i=1,2,\cdots,n$ and $\|x^{0}_{\varepsilon^{0}}(\omega)\|\leqslant \beta^{0}(\omega)+\varepsilon^{0}(\omega)$ a.s., further let $\Omega_{0}=(\bigcap_{i=1}^{n}\{\omega\in \Omega~|~f^{0}_{i}(\omega,x^{0}_{\varepsilon^{0}}(\omega))=\xi^{0}_{i}(\omega)\})\cap\{\omega\in \Omega~|~\|x^{0}_{\varepsilon^{0}}(\omega)\|\leqslant \beta^{0}(\omega)+\varepsilon^{0}(\omega)\}$, then $x^{0}_{\varepsilon^{0}}$ and $\Omega_{0}$ meet our needs.

Finally, when $B$ is separable it is clear that (3.1) implies (3.2), we will prove that (3.2) also implies (3.1) as follows.

Let $Q$ be a countable dense subset of $K$, $Q^{n}=$the self-product of $n$'s copies of $Q$ and $\Omega_{1}=\bigcap_{(r_{1},r_{2},\cdots,r_{n})\in Q^{n}}\{\omega\in \Omega~|~|\Sigma_{i=1}^{n}r_{i}\xi^{0}_{i}(\omega)|\leqslant \beta^{0}(\omega)\cdot\mbox{sup}\{|\Sigma_{i=1}^{n}r_{i}f^{0}_{i}(\omega,b)|~|~b\in B~\mbox{and}~\|b\|\leqslant1\}\}$. Then $P(\Omega_{1})=1$ by (3.2) and for each $\omega\in \Omega_{1}$ we have:
$$|\Sigma_{i=1}^{n}\alpha_{i}\xi^{0}_{i}(\omega)|\leqslant \beta^{0}(\omega)\cdot\mbox{sup}\{|\Sigma_{i=1}^{n}\alpha_{i}f^{0}_{i}(\omega,b)|~|~b\in B~\mbox{and}~\|b\|\leqslant1\}$$
for any given $\alpha_{1},\alpha_{2},\cdots,\alpha_{n}\in K$.

Thus for each $\omega\in \Omega_{1}$ and any $\lambda^{0}_{1},\lambda^{0}_{2},\cdots,\lambda^{0}_{n}\in \mathcal{L}^{0}(\mathcal{F},K)$ we also have that $|\Sigma_{i=1}^{n}\lambda^{0}_{i}(\omega)\xi^{0}_{i}(\omega)|\leqslant\beta^{0}(\omega)\cdot$sup$\{|\Sigma_{i=1}^{n}\lambda^{0}_{i}(\omega)f^{0}_{i}(\omega,b)~|~b\in B$ and $\|b\|\leqslant1\}$, namely (3.1) holds.
\end{proof}

\begin{Corollary}\label{cor4}
Let $(B,\|\cdot\|)$ be a separable normed space, $(\Omega,\mathcal{F},P)$ a complete probability space, $f^{0}_{1},f^{0}_{2},\cdots,f^{0}_{n}:\Omega\times B\rightarrow K$, $\xi^{0}_{1},\xi^{0}_{2},\cdots,\xi^{0}_{n}$ and
$\beta^{0}$ the same as in Corollary \ref{cor3}. Then for each real-valued $\mathcal{F}$-random variable $\varepsilon^{0}$ satisfying $\varepsilon^{0}(\omega)>0$ for each $\omega\in \Omega$ there exists a $B$-valued $\mathcal{F}$-random variable $x^{0}_{\varepsilon^{0}}$ such that the following two items hold:

(1) $f^{0}_{i}(\omega,x^{0}_{\varepsilon^{0}}(\omega))=\xi^{0}_{i}(\omega)$ for each $\omega\in \Omega$ and each $i$ such that $1\leqslant i \leqslant n$;

(2) $\|x^{0}_{\varepsilon^{0}}(\omega)\|\leqslant \beta^{0}(\omega)+\varepsilon^{0}(\omega)$ for each $\omega \in \Omega$

\noindent iff $
|\Sigma_{i=1}^{n}\alpha_{i}\xi^{0}_{i}(\omega)|\leqslant \beta^{0}(\omega)\cdot\mbox{sup}\{~|\Sigma_{i=1}^{n}\alpha_{i}f^{0}_{i}(\omega,b)|~|~b\in B~\mbox{and}~\|b\|\leqslant1\}$ for each $\omega\in \Omega$ and any $\alpha_{1},\alpha_{2},\cdots,\alpha_{n}\in K$.
\end{Corollary}

\begin{proof}
Necessity is clear.

Sufficiency. By Corollary \ref{cor3} there exist a $B$-valued $\mathcal{F}$-random variable $y^{0}_{\varepsilon^{0}}$
and an $\mathcal{F}$-measurable subset $\Omega_{0}$ of probability one such that the following two items hold for each $\omega \in \Omega_{0}$ and $i=1,2,\cdots,n$:

(3) $f^{0}_{i}(\omega,y^{0}_{\varepsilon^{0}}(\omega))=\xi^{0}_{i}(\omega)$;

(4) $\|y^{0}_{\varepsilon^{0}}(\omega)\|\leqslant \beta^{0}(\omega)+\varepsilon^{0}(\omega)$.

For each $\omega\in \Omega\setminus \Omega_{0}$, by the classical Helly theorem there exists $V(\omega)\in B$ such that the following two assertions hold:

(3) $f^{0}_{i}(\omega,V(\omega))=\xi^{0}_{i}(\omega)$ for each $i=1,2,\cdots,n$;

(4) $\|V(\omega)\|\leqslant \beta^{0}(\omega)+\varepsilon^{0}(\omega)$.

Define $x^{0}_{\varepsilon^{0}}:\Omega\rightarrow B$ by
$$
x^{0}_{\varepsilon^{0}}(\omega)=\left \{
\begin{array}{ll}
y^{0}_{\varepsilon^{0}}(\omega), &\mbox{when~}\omega\in \Omega_{0};\\
V(\omega), &\mbox{when~}\omega\in \Omega\setminus \Omega_{0}.
\end{array}
\right.
$$
Then $x^{0}_{\varepsilon^{0}}$ is still a $B$-valued $\mathcal{F}$-random variable since $(\Omega,\mathcal{F},P)$ is complete, and $x^{0}_{\varepsilon^{0}}$ satisfies both (1) and (2).
\end{proof}

\begin{remark}
When $(B,\|\cdot\|)$ is a separable Banach space Corollary \ref{cor4} can also obtained from the measurable selection theorem \cite[Theorem 5.10]{wagner}. In fact, the necessity of this corollary is always clear, for the sufficiency let us define $G:\Omega\rightarrow 2^{B}$ as follows:
$$G(\omega)=\{b\in B~|~\|b\|\leqslant \beta^{0}(\omega)+\varepsilon^{0}(\omega)~\mbox{and}~ f^{0}_{i}(\omega,b)=\xi^{0}_{i}(\omega)~\mbox{for each}~i=1,2,\cdots,n\}$$
\noindent for each $\omega \in \Omega$. Then $G$ is $\mathcal{F}$-measurable and $G(\omega)$ is closed and nonempty for each $\omega\in \Omega$ by the classical Helly theorem, further by \cite[Theorem 5.10]{wagner} $G$ has an $\mathcal{F}$-measurable selection $x^{0}_{\varepsilon^{0}}$, which satisfies our desires. But the general cases such as Corollary \ref{cor3} and \ref{cor4} can not be obtained from \cite[Theorem 5.10]{wagner}, which exhibits the power of Theorem \ref{thm2}. Besides, Theorem \ref{thm2} will mainly serve the future development of the theory of $RN$ modules.
\end{remark}












\end{document}